\documentclass[12pt]{amsart}
\usepackage[a4paper]{geometry}
\usepackage{amsmath,amssymb,amsthm,bbm,mathrsfs,url,xspace}
\usepackage[all]{xy}
\usepackage{array}
\usepackage{verbatim} 
\usepackage{hyperref}
\newtheorem{thm}{Theorem}[section]

\newtheorem{lem}[thm]{Lemma}
\newtheorem{cor}[thm]{Corollary}
\newtheorem{prop}[thm]{Proposition}

\theoremstyle{definition}

\newtheorem*{nota}{Notations}
\newtheorem{rmq}[thm]{Remark}
\newtheorem{ex}[thm]{Examples}

\numberwithin{equation}{section}
\def\Z{\mathbb Z}
\def\F{\mathbb F}
\def\Fp{{\mathbb F}_p}
\def\F2{{\mathbb F}_2}
\def\U{\mathcal U}

\def\T{\mathrm{T}}

\def\Ap{{\mathcal A}_p}

\def\U{{\mathcal U}}
\def\H{ {H}}
\def\ra{\rightarrow}
\def\Hom{\mathrm {Hom}}
\def\End{\mathrm {End}}

\DeclareMathOperator{\Ker}{Ker}

\def\mod{\mathrm{-mod}}
\def\proj{\mathrm{-proj}}
\def\GL{\mathrm {GL}}
\def\M{\mathrm {M}}
\def\I{\mathcal I }
\def\J{\mathcal J }
\def\Gr{\mathrm{Gr} }
\def\St{\mathrm {St}}
\DeclareMathOperator{\Res}{Res}
\DeclareMathOperator{\Ind}{Ind}
\DeclareMathOperator{\Inf}{Inf}
 \title[Lannes' $\T$ vs Harish-Chandra restriction]{Lannes'  $\T$ functor on injective unstable modules and Harish-Chandra restriction}
\author[V. Franjou]{Vincent Franjou}
\address{Universit\'e de Nantes \\ 
Laboratoire de Math\'ematiques Jean Leray
(UMR 6629 CNRS \& UN)}
\email{vincent.franjou@univ-nantes.fr}
\thanks{The first author wants to thank VNU-HUS for its support and hospitality in the spring 2015 when the content of the paper was first discussed.}
\author[Nguyen D. H. Hai]{Nguyen Dang Ho Hai}
\address{Department of Mathematics, College of Sciences, University of Hue, Vietnam}
\email{nguyendanghohai@gmail.com}
\author[L. Schwartz]{Lionel Schwartz}
\address{LAGA, UMR 7539 CNRS, LIA Formath Vietnam, USPC University.}
\email{schwartz@math.univ-paris13.fr}
\thanks{
The second and third  authors thank the VIASM for his warm hospitality while this work was in progress. The third author thanks the program KH3CF of USPC University for his support. Unfortunately this project, like many others, has been interrupted by government decision, leaving USPC's researchers and teachers worrying about research policy.}
\usepackage{color}

\begin{document}
	
\begin{abstract} 
In the 1980's, the magic properties of the cohomology of elementary abelian groups 
as modules over the Steenrod algebra 
initiated a long lasting interaction between topology and modular representation theory in natural characteristic.
The Adams-Gunawardena-Miller theorem in  particular, showed that their decomposition is governed by the modular representations of the semi-groups of square matrices. 
Applying Lannes' T functor on the summands $L_P:=\mathrm{Hom}_{\M_n(\Fp)}(P,H^*(\Fp)^n)$ defines an intriguing construction in representation theory. We show that $\T (L_P) \cong L_P \oplus H^*V_1 \otimes L_{\delta (P)}$, defining a functor $\delta$ from $\mathbb F_p[\M_n(\mathbb F_p)]$-projectives to $\mathbb F_p[\M_{n-1}(\mathbb F_p)]$-projectives. We relate this new functor $\delta$ to classical constructions in the representation theory of the general linear groups.
\end{abstract}

\date{\today}
\maketitle

%
\section{Introduction}
\subsection{Setting}
For a prime $p$ and a positive integer $n$, let $\GL_n\mod$ denote the category of right 
modules over the group ring $\Fp[\GL_n(\Fp)]$. 
If we let $V_n$ be the $n$-dimensional $\Fp$-vector space $(\Fp)^n$ with the left action of $\GL_n(\Fp)$, its cohomology with mod $p$ coefficients $\H^*V_n$
is an example of such a module.
Given a projective object $P$ in $\GL_n\mod$, put  
\[
M_P:=\Hom_{\GL_n\mod }(P,\H^*V_n) 
.\]
The cohomology $\H^*V_n$ is also a module over the mod $p$ Steenrod algebra $\Ap$, and a peculiar one: 
it is an injective object  in $\U$, the category of unstable modules over $\Ap$ (we refer to the exposition in \cite{Sch94}). 
By naturality, 
the two actions on $\H^*V_n$ commute, 
making $M_P$ an object in $\U$.
Being a direct summand in an $\U$-injective, 
$M_P$ is also an injective object in the category $\U$.

Similarly, let $\M_n\mod$ denote the category of right modules over the semigroup ring $\Fp[\M_n(\Fp)]$ of $n\times n$ matrices over the field $\Fp$. 
Given a projective object $P$ in $\M_n\mod$, put  
\[
L_P:=\Hom_{\M_n\mod }(P,H^*V_n)
.\]
Again, 
$L_P$ is an injective object in the category $\U$. 
By the Adams-Gunawardena-Miller theorem, $\End_\U(\H^*V_n)\cong\Fp[\M_n(\Fp)]$, 
and standard arguments lead to the following statement.
\begin{thm} \cite{HK88} \label{HK88}
The correspondence 
$$
P \mapsto L_P=\Hom_{\mathrm \M_n\mod}(P,\H^*V_n)
$$
defines a contravariant equivalence 
from the category $\M_n\proj$ of projective $\Fp[\mathrm \M_n(\Fp)]$-modules, 
to the category $\mathcal I \U_n^{red}$ of injective unstable modules which are direct sums
of indecomposable factors of $H^*V_n$,
with inverse functor $\Hom_\U(-,H^*V_n)$. 
\end{thm}
Also, if $P$ is an indecomposable projective $\Fp[\M_n(\Fp)]$-module, 
then $L_P$ is an indecomposable injective in $\U$.
It is shown in \cite{LS89} that any reduced injective in $\U$ is a direct sum 
of such indecomposable summands
(an unstable module is called \emph{reduced} if it does not contain a non-zero suspension).
\subsection{Main result}
Lannes' main tool is his magic functor $\T: \U\to \U$,
a left adjoint to the functor given by  $M\mapsto M\otimes H$ where $H:=\H^*\Z/p$ \cite{Lan92}. 
The exactness of Lannes' T-functor reflects the $\U$-injectivity of $H$.
It comes with a reduced version $\overline{\T} $, a left adjoint to the tensor product with $\overline{H}=\overline{\H}^*\Z/p$, and the decomposition $H=\overline{H}\oplus \Fp$ induces a natural decomposition: $\T M \cong M\oplus \overline \T(M)$ 
Our main result describes the value of $T$ on reduced $\U$-injectives.
\begin{thm}\label{delta}
Let $E$ be a reduced unstable injective module. Then there exists a 
(up to $\U$-isomorphism)
uniquely defined reduced unstable injective module $\delta (E)$
such that:
\[
\overline \T (E) \cong H \otimes \delta (E)
\]
Moreover, when $E$ is an object of $\mathcal I \U_n^{red}$, then 
$\delta(E)$ is an object in $\mathcal I \U_{n-1}^{red}$.
\end{thm}
The second author proved the result for virtual modules by using topological methods in \cite{Hai15}. As explained there, this in turn implies the third author's conjecture about the eigenvalues of Lannes' T-functor (see also \cite{Hai16} for another proof of this conjecture).

\subsection{}
We now explain the core of the argument developed in Section \ref{H-mod:section}.

A classical computation in \cite{Lan92} gives an $\End (V)$-isomorphism of unstable modules
\[\overline{\T} \H^*(V)\cong \J (V)\otimes\H^*V ,\]
where $\J (V)$ denotes the quotient of the vector space of set maps, $\Fp^V$, by the sub-space of constant maps.
Hence a natural isomorphism, for $P$ in $\M_n\proj$:
\begin{equation}\label{mu}
\overline{\T} (L_P) \cong \Hom_{\M_n\mod}(P, \J (V_n)\otimes \H^*V_n).
\end{equation}
This is a reduced unstable module for each $P$, indeed a submodule in direct sums of copies of $\H^*V_n$, 
and it is $\mathcal Nil$-closed \cite{Sch94}.   

Instead of $\J (V)$, we may consider its Kuhn dual $$\I(V):= \J(V^*)^*,$$ 
 which is the kernel of the augmentation in the group algebra of $V^*$.
We define another right-exact functor from projective $\M_n$-modules to unstable modules by the formula
\[
\nu_n(P):= \Hom_{\M_n\mod }(P, \I (V_n)\otimes \H^*V_n)
 .\]
A key observation of this paper is the following.
\begin{prop}\label{main1} 
For any $P$ in $\M_n\proj$, 
there is an isomorphism of unstable modules: 
\begin{equation*}\label{iso}
\overline{\T} (L_P)  \cong \nu_n(P).
\end{equation*}
\end{prop}

The point is that  the right-hand side bears more structure.
Recall from \cite{LZ95} that an unstable $H$-module is an unstable module provided with an $H$-module structure, for which the Cartan formula holds.

\begin{lem}\label{H-mod} The functor $\nu_n$ takes values in the category of free unstable $H$-modules. 
\end{lem}

As a consequence,
the unstable  module $\overline{\T} (L_P)$
is also a free unstable $H$-module.
Theorem \ref{delta} then follows from a theorem of Bourguiba \cite{Bou09} . 
%

\begin{rmq}Note that the functors $P\mapsto L_P$ and $P\mapsto \nu_n(P)$ extend to all $\M_n\mod$, and  Lemma \ref{H-mod} is still valid if we consider $\nu_n$ as a functor from $\M_n\mod$. However Proposition \ref{main1} is no longer true in general.
\end{rmq}

\subsection{Properties of $\delta$} 
The next properties follow easily from the properties of the functor $\overline \T$.
\begin{prop}\label{prop}
Let $E$ be a direct summand of $\H^*V_r$ and $F$ be a direct summand of $\H^*V_s$. Then
\[
\delta_{r+s}(E \otimes F) \cong  (E \otimes \delta_s(F)) \oplus (\delta_r(E) \otimes F) \oplus (H \otimes \delta_r(E) \otimes \delta_s(F)).
\]
Moreover there is a  functor $\omega \colon \mathcal I \U_n^{red} \ra \mathcal I \U_n^{red}$ such that $\omega (H) =H$, $\omega (\Fp)=\{0\}$ and
\[
\delta_n \circ \omega (E) \cong  \delta_n(E) \oplus \omega \circ \delta_n (E). 
\]
\end{prop}
We now show that $\delta$ is compatible with the filtration by the dimension $n$.
\begin{thm} \label{functorEnd}The  exact functor
$\Fp \otimes_H \nu_n(-)$  induces 
an exact functor 
\[
\delta_n : \ \M_n\proj \ra \M_{n-1}\proj 
,\]
such that, for each $P$ in $\M_n\proj$, there is an isomorphism of unstable modules:
\[
\overline{\T} (L_P) \cong H\otimes L_{\delta_n(P)}
.\]
\end{thm}
A similar result holds for projective $\Fp[\GL_n(\Fp)]$-modules. 
\begin{thm} \label{functorGL}
The  exact functor 
\[
P\mapsto \Fp \otimes_H  \Hom_{\GL_n\mod }(P, \I (V_n)\otimes \H^*V_n)
\]
from $\GL_n\proj$ takes values in $\U_{n-1}^{red}$. 
It induces an exact functor from $\GL_{n}\proj$ to $\GL_{n-1}\proj $, 
which we denote again by $\delta_n$,
such that for each $P$ in $\GL_n\proj$, there is an isomorphism of unstable modules:
\[
\overline{\T} (M_P) \cong H\otimes M_{\delta_n(P)}
.\]
\end{thm}
%
%
\begin{ex}
Let $\St_n$ denote the Steinberg representation of $\GL_n(\Fp)$ \cite{Ste56}. 
For $\rho$ in $\GL_n\mod$
we denote by $P_\rho$ its projective cover in $\GL_n\proj$.
For these examples, we let $p=2$.
\begin{enumerate}
\item $\delta_3 (P_{\F2})= 2P_{\F2}\oplus 2 \St_2$,
\item $\delta_3 (P_{V_3})= P_{\F2}\oplus  \St_2$,
\item $\delta_3 (P_{V_3^\#})= P_{\F2}\oplus  \St_2$,
\item $\delta_n(\St_n)=\St_{n-1}$, $n\geq 1$.
\end{enumerate}
\end{ex}
For $p=2$, an isomorphism $\overline \T M_n\cong H\otimes M_{n-1}$, was first obtained by John Harris and James Shank by combining their work \cite{Harris-Shank-92} with work of Dave Carlisle and Nick Kuhn \cite{CK96}.
Here $M_n\cong L_n\oplus L_{n-1}$ is the Steinberg unstable module of Stephen Mitchell and Stewart Priddy \cite{MP83}, which is the same as $M_{\St_n}\cong L_{\St_n}\oplus L_{\St_{n-1}}$ in our notation
(indeed we kept the letters $M$ and $L$).
Their result also writes in our notations:
\[\delta (L_{\St_n})=L_{\St_{n-1}}\ ,\ \ \ \delta (M_{\St_n})=M_{\St_{n-1}}\ \ \textrm{or} \ \ \delta_n({\St_n})=\St_{n-1}
\,. \]
We generalize this early result. 
The following theorem gives a formula for the action of Lannes' $\T$-functor on the unstable module $M_{\St_n\otimes X}$.
Note that for every $X$ in $\GL_n\mod$, because $\St_n$ is a projective in $\GL_n\mod$, the tensor product $\St_n\otimes X$ is also projective in $\GL_n\mod$. 
\begin{thm}\label{main2}For $X\in \GL_n\mod$, there is an isomorphism in $\GL_{n-1}\proj$: 
\[
\delta_n(\St_n\otimes X) \cong \St_{n-1}\otimes \Res_{\GL_{n-1}}^{\GL_n}X
\  .\]
\end{thm}
\begin{rmq}
A theorem of George Lusztig \cite{Lusztig76} (also due to John W. Ballard) states that a projective object of $\GL_n\mod$ can always be written as $\St_n\otimes X$ where $X$ is a virtual object in $\GL_n\mod$.
One can deduce from this result and from Theorem \ref{main2} that, for $P$ a projective object in $\GL_n\mod$, $\overline \T(M_{P})$ is of the form $H\otimes M'$ where $M'$ is a \emph{formal} sum of direct summands of $\H^*V_{n-1}$.
This provides yet another proof of Theorem \ref{functorGL} for virtual projectives.
\end{rmq}
The above examples support the following:
\begin{thm}\label{HC:thm}
For a $\GL_n(\Fp)$-projective module $P$,
the $\GL_{n-1}(\Fp)$-projective $\delta_n(P)$ 
is isomorphic to the Harish-Chandra restriction of $P$
for the Levi subgroup $\GL_{n-1}(\Fp)\times\GL_1(\Fp)$,
restricted to $\GL_{n-1}(\Fp)$.
\end{thm}
We first learned the result as a consequence of Theorem \ref{main2} by comparing characters.
The proof given in Section \ref{HC:section}
uses the formula for $\delta$ proven in Section 2 (Formula (\ref{deltaM:equation})). 
%
\begin{nota}
As usual, we let $\H^*V$ be the mod $p$ cohomology of the elementary abelian group $V$.
We denote by $H$ the cohomology
$\H^*\Z/p$,
$H\cong \Fp[\beta t] \otimes \Lambda(t)$, $\vert t \vert =1$, 
for $p$ odd,
(resp. $H = \H^*\Z/2 \cong \F2[t]$, for $p=2$),
and by $\overline H$ the reduced cohomology.
Note that $\H^1V$ identify with $V^*$, so that $\mu^*(t)$ identifies with $\mu$.

The functor $\T \colon \U \ra \U$ is left adjoint to the tensor product by $H$. 

We let $\J (V)$ be the quotient of the vector space of set maps $\Fp^V$ by the sub-space of constant maps,
and we let $\I(V)$ be the augmentation ideal in $\Fp [V^*]$. 

Categories include:
 the category $\GL_n\mod$ of right $\Fp[\GL_n(\Fp)]$-modules and its full subcategory $\GL_n\proj$ of projective $\Fp[\GL_n(\Fp)]$-modules;
 the category $\M_n\mod$ of right $\Fp[\M_n(\Fp)]$-modules and its full subcategory $\M_n\proj$ of projective $\Fp[\GL_n(\Fp)]$-modules;
the category   $\U$ of unstable modules,
and its full subcategory  $\mathcal I \U_n^{red}$  of injective unstable modules which are direct sums
of indecomposable factors of $H^*V_n$.
\end{nota}
%
\section{\textit{H}-module structure, proof of \ref{delta} to \ref{functorEnd}}
\label{H-mod:section}

\subsection{Proof of Proposition \ref{main1}}\label{proof of main1:subsection}
Let $V$ be an elementary abelian $p$-group.
The computation of $\T \H^*(V)$ as an unstable $\End (V)$-module is due to Lannes \cite{Lan92} (see also \cite[3.9.1]{Sch94}):
\[
\T \H^*(V)\cong \H^*(V)^V\cong \Hom_{\Fp}(\Fp[V],\H^*(V))\cong \Fp^V\otimes\H^*(V)
\,.
\]
The right $\End (V)$-action on the right-hand side is given by: 
\[
(\phi \otimes  x).\varphi = (\phi \circ \varphi) \otimes \varphi^*(x)
\,.
\]
For the reduced version:
\begin{thm}\label{barT} 
Let $\J (V)$ be the quotient of the vector space of set maps $\Fp^V$ by the sub-space of constant maps.
There is a $\End(V)$-equivariant isomorphism of unstable modules:
\[
\overline{\T}   \H^*(V) \cong 
\J (V)\otimes \H^*V
\]
\end{thm}
This description of $\overline{\T}   \H^*(V)$ allows an easy access to the values of the exact functor $\overline{\T}  $ on $\U$-injectives. 
Indeed, if $P$ is a projective $\Fp[\End(V)]$-module:
\begin{equation}\label{TL}
\overline{\T}   (L_P)
	=\Hom_{\End(V)}(P,\overline{\T}    ( \H^*V) )
	=\Hom_{\End(V)}(P,\J (V)\otimes \H^*V )
,
\end{equation}
and if $P$ is a projective $\Fp[\GL(V)]$-module:
\begin{equation}\label{TM}
\overline{\T}   (M_P)
	=\Hom_{\GL(V)}(P,\overline{\T}    ( \H^*V) )
	=\Hom_{\GL(V)}(P,\J (V)\otimes \H^*V )
.
\end{equation}
Though the $\GL(V)$-structure on $\overline{\T}    ( \H^*V) $ is essential here, this formula does not use its full complexity.
One gains flexibility by using the following facts from representation theory:
\begin{enumerate}
\item two $\GL(V)$-projectives with isomorphic Jordan-H\"{o}lder subquotients are isomorphic \cite[\S 16.1, Corollaire 2 du Th\' eor\`eme 35]{Serre77};
\item tensoring with a $\GL(V)$-projective yields a $\GL(V)$-projective;
\item tensoring with a finite-dimensional $\GL(V)$-module $M$ is adjoint to tensoring with its contragredient $M^\#$.
\end{enumerate}
As a result:
\[
\Hom_{\GL(V)}(P,M\otimes \H^*V )\cong \Hom_{\GL(V)}(P\otimes M^\#,\H^*V )
= M_{P\otimes M^\#}
\]
only depends on the Jordan-H\"{o}lder subquotients of $M$
(this will be used also in the last section). 
Moreover in the above formula (\ref{TM}), one can replace $\J(V)$ by its contragredient, that is the augmentation ideal $\I(V)$, 
or its associated graded $\Gr(V)$. 
Indeed, when $p=2$, $\Gr(V)$ is the sum of the exterior powers $\Lambda^kV^*$, $k>0$; when $p>2$ it is the sum of the  
reduced symmetric powers $S^k(V^*)/(\mu^p)$, quotient of the symmetric powers
by the ideal generated by $p$-th powers. 
In both cases, these are given by semi-simple functors on $V$,
and as these are self-dual, so is $\Gr(V)$.
Hence there are isomorphisms of unstable modules:
\begin{equation}
 \Hom_{\GL(V)}(P, \J(V) \otimes \H^*V) \cong  
 \Hom_{\GL(V)}(P,  \I(V) \otimes \H^*V).
\end{equation}
However the isomorphisms are (in general) not natural in $P$.
This proves the $\GL$ version of Proposition \ref{main1}:
\begin{prop}
For any $P$ in $\GL_n\proj$, 
there is an isomorphism of unstable modules: 
\begin{equation*}
 \overline{\T} (M_P) \cong \Hom_{\GL_n)}(P,  \I(V_n) \otimes \H^*V_n)
.\end{equation*}
\end{prop}
\medskip
Even though we cannot use the contragredient, 
we now show that one still can argue similarly for the $\End(V)$ case.

Consider a short exact sequence of $\End(V)$-modules:
$
0\ra K \ra E \ra Q \ra 0
.$
Using projectivity of  $P$ and injectivity of the unstable modules, one gets a splitting:
\[
 \Hom_{\End (V) } (P, E\otimes \H^*V) 
 \cong  \Hom_{\End (V) } (P, K \otimes \H^*V) \oplus \Hom_{\End (V) } (P, Q \otimes \H^*V).
\]
Thus:
\[
 \Hom_{\End (V) } (P, \I(V) \otimes \H^*V) 
 \cong  \Hom (P, \Gr(V) \otimes \H^*V)
 \cong  \Hom_{\End (V) } (P,  \J(V) \otimes \H^*V)
\]
where $\Gr(V)$ is the direct sum of the $\End(V)$-composition factors, which is easily seen to be  the same as above.

The direct sum decomposition depends on choices and, again, the isomorphisms are not natural in $P$.

This concludes the proof of Proposition \ref{main1}.
\subsection{The \textit{H}-module structure }\label{H:subsection}
We now describe the $H$-module structure on $\I (V)\otimes \H^*V$.
\begin{prop} \label{H-module:prop}
The unstable module $\I(V)\otimes \H^*V$
is endowed with a free $H$-module structure defined by:
\[
u.((\mu) \otimes x)= (\mu) \otimes \mu^*(u) x
,\
\mu\in V^*,\ (\mu)=[\mu]-[0] \in \I(V),\ 
u \in H,\ x\in\H^*(V)
.\]
The $H$-action is $\End(V)$-equivariant.
Moreover 
the quotient 
$
\Fp\otimes_H ( \I(V)\otimes \H^*V )
$
is a reduced unstable injective module.
\end{prop}
\begin{proof}
For each $\mu$ in $V^*$, the $H$-structure on $\H^*(V)$ is induced by the short exact sequence
\[
0\to\Ker\mu\subset V\rightarrow\Fp\to 0.
\]
Hence, the $H$-structure is free and
 $\Fp\otimes_H \H^*(V)\cong \H^*(\Ker\mu)$.
Assembling these gives an 
isomorphism of unstable modules
\begin{equation}\label{formula:equation}
\Fp\otimes_H (\I(V)\otimes \H^*V)\to \bigoplus_{\mu \in V^*\setminus\{0\}}\H^*(\Ker \mu)
\end{equation}
sending $1\otimes((\mu)\otimes x)$ to $\mu^*(x)$ in the summand $\H^*(\Ker \mu)$.
The $\End(V)$-equivariance results from:
 \begin{align*}
\big(u. \left( \left(\mu\right) \otimes x \right) \big) .\varphi & =
((\mu) \otimes \mu^*(u) x). \varphi =
 (\mu \circ \varphi) \otimes \varphi^*(  \mu^*(u) x)
 \\
 &=
 (\mu \circ \varphi) \otimes  \varphi^*[ \mu^*(u)] \varphi^*(x)=
 (\mu \circ \varphi) \otimes   (\mu \circ \varphi)^*(u) \varphi^*(x)
\\&=
 u.\big( \left( (\mu) \otimes x \right) . \varphi \big)
 \end{align*}
The proposition is proved. 
\end{proof}
The quotient $\bigoplus_{\mu \in V^*\setminus\{0\}}\H^*(\Ker \mu)$
inherits a right $\End(V)$ action. Explicitly, 
for $\varphi$ in $\End(V)$ and $x$ in a summand $\H^*(\Ker \mu)$,
$x.\varphi=0$ if $\mu\circ\varphi=0$, 
and $x.\varphi=\varphi^*(x)$ in the summand $\H^*(\Ker (\mu\circ\varphi))$ if $\mu\circ\varphi\neq 0$.
\begin{cor}
For $P$ in $\End(V)\proj$ (resp. in $\GL(V)\proj$),
there is a natural free $H$-module structure on
$\overline{\T}   (L_P)$
and an isomorphism of unstable modules:
\begin{eqnarray}
\Fp\otimes_H \overline{\T}   (L_P) &\cong \Hom_{\End (V) } (P, \bigoplus_{\mu \in V^*\setminus\{0\}}\H^*(\Ker \mu))
\\
\textrm{ (resp. }
\Fp\otimes_H \overline{\T}   (M_P) &\cong \Hom_{\GL (V) } (P, \bigoplus_{\mu \in V^*\setminus\{0\}}\H^*(\Ker \mu))\, )
\end{eqnarray}
These unstable modules are $\mathcal Nil$-closed.
\end{cor}
%
\subsection{Proof of \ref{delta}, \ref{H-mod}, and \ref{functorEnd}} 
We now use D. Bourguiba's work \cite{Bou09} and prove Theorem \ref{delta} to conclude:
\[
\overline{\T}   (L_P) \cong H \otimes \delta (L_P) \ \  \textrm{ (resp. }\overline{\T} (M_P) \cong H \otimes \delta (M_P) \textrm{\,),}
\]
for projectives $P$.
\begin{thm}(\cite[Theorem 3.2.1]{Bou09}) 
Let $E$  be an unstable $H-\Ap$-module which is free as an $H$-module and let $\mathcal E(E)$ be its injective hull (in the category $H-\U$).
We suppose that $\Fp \otimes _H E$ is reduced and let $I$ be its injective hull in the category $\U$.
Then $\mathcal E (E)$ is isomorphic, as an unstable $H-\Ap$-module, to $H \otimes I$.
\end{thm}
In the reference, the result is proved for the prime $2$, but its proof adapts for odd primes. 
Indeed, Lemma 3.1.1 in \cite{Bou09} only uses that $Q$ is $H$-free, 
and Proposition 2.6.2 only uses that the module is torsion-free over the sub-polynomial algebra $\Fp[\beta t]$ of $H$.

We need the following particular case of the theorem 
\cite[Remark 4.3]{Bou09}:\\
\emph{If $ \overline E = \Fp \otimes _{H} E$  is a reduced injective object in $\U$,
then the only unstable free $H-\Ap$-module $M$, up to $\U$-isomorphism, such that 
$\Fp \otimes _{H} M \cong \overline E$ is $H \otimes \overline E$.}
\\
We thus get an isomorphism of unstable modules:
$\overline{\T}   (L_P)\cong H\otimes \delta(L_P)$.
The unstable module 
\[
\delta(L_P)=\Hom_{\End(V)}\big( P, \bigoplus_{\mu \in V^*\setminus\{0\}}\H^*(\Ker \mu)\big)
\]
is injective, and more precisely
it is a summand in $\bigoplus_{\mu \in V^*\setminus\{0\}}\H^*(\Ker \mu)$. 
A similar formula holds with $\GL(V)$, see Section \ref{HC:section}.

To prove Theorem \ref{functorEnd}, there remains to show that if $P$ is in $\M_n\proj$,
then $\delta (L_P) \cong L_Q$, for $Q$ in $\M_{n-1}\proj$. 
This follows 
from Theorem \ref{HK88}. 
\section{Proof of Proposition \ref{prop}}
For the first part of Proposition \ref{prop}, apply the definitions.
For the second part, 
the functor $\omega$
is induced by the functor $ \Omega$ defined on unstable modules
as left adjoint to suspension.
The details of the argument will be discussed elsewhere. 
%
\section{The functor $\delta_n$ on representations of the General Linear group}
\label{HC:section}
We now concentrate on the $\GL$ case to relate our construction to classical representation theory of the General Linear group.
\subsection{Proof of Theorem \ref{functorGL}}\label{GL:subsection}
Recall from Section \ref{H-mod:section} that for $n\geq 1$ and $P$ in $\GL_n\proj$, 
$\overline{\T} (M_P) \cong H \otimes \delta (M_P)$
for
\begin{equation}\label{deltaM:equation}
\delta(M_P)=\Hom_{\GL_n(\Fp)}\big( P, \bigoplus_{\mu \in {V_n}^*\setminus\{0\}}\H^*(\Ker \mu)\big)
.
\end{equation}
We now show how the $\GL_n(\Fp)$ representation 
$\bigoplus_{\mu \in {V_n}^*\setminus\{0\}}\H^*(\Ker \mu)$
can be induced from the $\GL_{n-1}(\Fp)$ representation $\H^*(V_{n-1})$.

The group $\GL_n=\GL_n(\Fp)$ is the automorphism group of the vector space $V_n=\Fp^n$. Take the canonical basis $(e_1,\dots ,e_n)$, and
choose $V_{n-1}$ to be the hyperplane of last coordinate zero in $V_n$. 
We refer to the subgroup of $\GL_n$ which stabilizes $V_{n-1}$
as the parabolic subgroup.
It is a semi-direct product $LU$ of two subgroups: the Levi subgroup $L:=\GL_{n-1}\times\GL_1$, and the unipotent subgroup $U$ of matrices $g$ in $\GL_n$ acting as the identity on $V_{n-1}$ and such that $g(e_n)-e_n$ is in $V_{n-1}$. 
The subgroup $\GL_{n-1}=\GL_{n-1}(\Fp)$ is the subgroup $\GL_{n-1}\times 1$ of the Levi subgroup. In other words, the group $\GL_{n-1}$ is considered as a subgroup of $\GL_n$ \emph{via} the inclusion $g\mapsto \begin{pmatrix}
g & 0\\ 0 & 1
\end{pmatrix}$.

Recall that inflation is the simplest way to extend a representation from a subgroup $H$ to a semi-direct product $HU$, by just letting the elements of the normal subgroup $U$ act by the identity.
We denote the resulting exact functor by $\Inf_H^{HU}$.
We note that its left adjoint is given by taking coinvariant of the $U$-action.
\begin{prop}\label{induction:prop}
There is an isomorphism of $\Ap[GL_n]$-modules:
\begin{equation}\label{induction:equation}
\bigoplus_{\mu \in {V_n}^*\setminus\{0\}}\H^*(\Ker \mu)
\cong	
\Ind_{GL_{n-1}U}^{GL_n} \Inf_{GL_{n-1}}^{GL_{n-1}U} H^*V_{n-1}  
	.
\end{equation}
\end{prop}
\begin{proof}
Start with the inclusion
\[
\H^*V_{n-1} =\H^*\Ker \mu_0\hookrightarrow \bigoplus_{\mu\not =0} \H^*\Ker\mu 
\]
of $\H^*V_{n-1}$ as the factor indexed by the n-th coordinate form $\mu_0=e^*_n$.
Every element of $U$ fixes $\mu_0$, and is the identity on $V_{n-1}$.
Thus, the inclusion inflates to a $\GL_{n-1}U$-equivariant map:
\[
\Inf_{\GL_{n-1}}^{\GL_{n-1}U}H^*V_{n-1} \hookrightarrow \bigoplus_{\mu\not =0} H^*\Ker\mu 
.\]
Through adjunction between restriction and induction, this map corresponds to a $\GL_n$-equivariant map 
\[
\Fp[\GL_n]\bigotimes_{\Fp[\GL_{n-1}U]} \H^*V_{n-1}\cong 
\Ind_{\GL_{n-1}U}^{\GL_n} \Inf_{\GL_{n-1}}^{\GL_{n-1}U} \H^*V_{n-1} 
\to \bigoplus_{\mu\not =0} \H^*\Ker\mu 
.\]
This map sends $g\bigotimes x$ to $g^*(x)$ in $\H^*\Ker(\mu_0\circ g)$.
As $GL_n$ acts transitively on the set 
$\{ \Ker (\mu ) \mid \mu \in {V_n}^*\setminus\{0\}\}$,
the resulting map is surjective.
Since the vector space on the left-hand side is isomorphic to  $|GL_n|/|\GL_{n-1}U| =p^n-1$ copies of $H^*V_{n-1}$, 
dimensions agree in each degree, and
the map is bijective.	
\end{proof}
It follows from (\ref{deltaM:equation}) and (\ref{induction:equation}) 
that
\[
\delta(M_P)=\Hom_{\GL_n(\Fp)}\big( P, \Ind_{GL_{n-1}U}^{GL_n} \Inf_{GL_{n-1}}^{GL_{n-1}U} H^*V_{n-1}  \big)
\]
Note that induction, as inflation, is exact, and that the left adjoint of an exact functor sends projectives to projectives.
So the unstable module $\delta(M_P)$ is isomorphic to $M_Q$, 
where $Q$ is obtained from $P$ by applying the left adjoint of $\Ind_{\GL_{n-1}U}^{\GL_n} \Inf_{\GL_{n-1}}^{\GL_{n-1}U} $.
Explicitly, 
$Q$ is the $\GL_{n-1}$-projective $(\Res_{\GL_{n-1}U}^{GL_n} P)_U$.
This proves Theorem \ref{functorGL}.
\subsection{Harish-Chandra restriction and proof of Theorem \ref{HC:thm}}
We refer to \cite{DM91}. The presentation there is made over the complex numbers, but it is adequate in any characteristic.

Let us first describe Harish-Chandra induction.
Starting with a
$\GL_{n-1}\times\GL_1$-module,
inflate the action to the parabolic (by letting $U$ act by the identity); 
then apply induction to $\GL_n$.
Note that both steps are exact, so Harish-Chandra induction is exact. 
When $p=2$, the Levi subgroup coincide with $\GL_{n-1}$, and we recover the construction from Proposition \ref{induction:prop}.

Harish-Chandra restriction is defined as the adjoint of Harish-Chandra induction. 
There is a choice of right or left adjoint here,
and for our purpose,
we define Harish-Chandra restriction as the left adjoint.
Explicitly, starting with a $\Fp[\GL_n(\Fp)]$-module $X$, 
the Harish-Chandra restriction of $X$ is obtained by
restriction to the parabolic subgroup,
followed by 
taking coinvariants of the corresponding unipotent $U$:
\[
(\Res_{LU}^{GL_n} X)_U
.\]

To compare with the construction in Proposition \ref{induction:prop},
note that, as $L/\GL_{n-1}\cong LU/\GL_{n-1}U\cong \GL_1$, we have:
\[
\Inf_{L}^{LU} \Ind_{\GL_{n-1}}^{L} 
\cong
\Ind_{\GL_{n-1}U}^{LU} \Inf_{\GL_{n-1}}^{\GL_{n-1}U} 
.\]
Taking adjoints, this proves Theorem \ref{HC:thm}.
\subsection{Proof of Theorem \ref{main2}}
We deduce Theorem \ref{main2} from the known properties of Harish-Chandra restriction.

We want to compare
$\delta_n(\St_n\otimes X)$ with $\St_{n-1}\otimes \Res_{\GL_{n-1}}^{\GL_n}X$.
For this, we can apply the character viewpoint used in \cite{DM91}:
indeed, projectives representations are characterized by their Brauer character.
In particular we want to use \cite[Corollary 7.4]{DM91}, which states exactly what we need on characters to prove Theorem \ref{main2}.
The proof of  \cite[Corollary 7.4]{DM91} is the averaging formula
\[
\vert U\vert^{-1}\sum_{u\in U}f(\ell u)
\] 
and it does not go through in positive characteristic, because of the division by $\vert U\vert$;
however, it can be replaced by the much simpler $f(\ell)$.
\subsection{A direct proof of Theorem \ref{main2} and \ref{HC:thm}}
Our original proof of Theorem \ref{main2} is also based on character comparison.
The key lemma here is the following.
\begin{lem}\label{key}There is an isomorphism in $\GL_n\mod$: 
\[
\Ind_{\GL_{n-1}}^{\GL_n}(\St_{n-1})\cong \J(V_n)^\#\otimes\St_n
.\]
\end{lem}
\begin{proof}
	

	Fix a $p$-regular element $s\in \GL_n$. 
	One has 
	$$\Ind_{\GL_{n-1}}^{\GL_n}\St_{n-1}(s) =\frac{1}{|\GL_{n-1}|}\sum_{g\in G \atop gsg^{-1}\in \GL_{n-1}}\St_{n-1}(gsg^{-1}).$$
	Here for $R$ an $\Fp[\GL_n(\Fp)]$-module, $R(s)$ denote the evaluation of the Brauer character of $R$ at $s$. 
	Suppose that $\dim_{\Fp}\Ker (s-1)=i$. In order to prove the identity $$\Ind_{\GL_{n-1}}^{\GL_n}\St_{n-1}(s)=\\St_n(s)\times \mathcal J(V_n)^\#(s)$$ 
	one needs the following three facts: 
	\begin{enumerate}
		\item For $g\in \GL_n$ such that $\sigma:=gsg^{-1}\in \GL_{n-1}$,  $\St_n(\sigma)=p^{i-1}\St_{n-1}(\sigma)$, $i\ge 1$. 
		\item $\mathcal J(V_n)^\#(s)=p^i-1$, $i\ge 0$.
		\item $\#\{g\in \GL_n\mid gsg^{-1}\in \GL_{n-1}\}=|\GL_{n-1}|(p^i-1)p^{i-1}$, $i\ge 0$.
	\end{enumerate}
	
The first one \footnote{This fact was used by O. Brunat to show that the characters of $\Res_{\GL_{n-1}}^{\GL_n} \St_n$ 
	and  $\St_{n-1}\otimes \omega_{n-1}$ are the same, where $\omega_{n}$ denotes the Weil representation of $\GL_n(\Fp)$.} is proved by Oliver Brunat in  \cite[Appendix]{HZ09}. 

For the second fact, recall that $\Fp[V_n^*]$ and the truncated symmetric algebra $S^*(V_n^*)/(x^p)$ have the same composition factors, 
and that the Brauer character of $S^*(V_n^*)/(x^p)$ evaluated at a $p$-regular element having $1$ as an eigenvalue of multiplicity $i$ is $p^i$. 
It follows that $\mathcal J(V_n)^\#(s)=p^i-1$.

We leave the details of the counting for the third fact to the reader.
\end{proof}
We now deduce Theorem \ref{main2} by the following sequence of isomorphisms of unstable modules:
\begin{eqnarray*}
\overline \T(M_{\St_n\otimes X})
& \cong &\Hom_{\GL_n}\big(\St_n\otimes X,  \J(V_n) \otimes \H^*V_n\big)\\
&\cong &  \Hom_{\GL_n}\big(\J(V_n)^\#\otimes\St_n,X^\# \otimes \H^*V_n\big)\\
&\cong &  \Hom_{\GL_n}\big(\Ind_{\GL_{n-1}}^{\GL_n}(\St_{n-1}),X^\# \otimes \H^*V_n\big)\\
& \cong & 
\Hom_{\GL_{n-1}}\big(\St_{n-1}, \Res_{\GL_{n-1}}^{\GL_n}(X^\# \otimes \H^*V_{n})\big)\\
&\cong &
\Hom_{\GL_{n-1}}\big(\St_{n-1}, \Res_{\GL_{n-1}}^{\GL_n}(X^\#) \otimes \H^*V_{n-1}\otimes H\big)\\
& \cong & 
\Hom_{\GL_{n-1}}\big (\St_{n-1}, \Res_{\GL_{n-1}}^{\GL_n}(X^\#)\otimes \H^*V_{n-1}\big )\otimes H\\
&\cong &
 \Hom_{\GL_{n-1}}\big(\St_{n-1}\otimes  \Res_{\GL_{n-1}}^{\GL_n}X,\H^*V_{n-1}\big)\otimes H\\
&\cong & M_{\St_{n-1}\otimes  \Res_{\GL_{n-1}}^{\GL_n}X} \ \otimes H.
\end{eqnarray*}
\begin{rmq}
We can deduce Theorem \ref{HC:thm} from Theorem \ref{main2} and \cite[Corollary 7.4]{DM91}. 
Indeed, these two results show that the statement is correct when $P$ is of the form $\St_n\otimes X$.
Since, by Lusztig's result \cite{Lusztig76}, every projective can be written $\St_n\otimes X$,
let us allow virtual $X$,
the statement is valid for every projective.
\end{rmq}

\providecommand{\bysame}{\leavevmode\hbox to3em{\hrulefill}\thinspace}
\providecommand{\MR}{\relax\ifhmode\unskip\space\fi MR }
\providecommand{\MRhref}[2]{%
	\href{http://www.ams.org/mathscinet-getitem?mr=#1}{#2}
}
\providecommand{\href}[2]{#2}

\end{document}